\documentclass[a4paper]{amsart}

\usepackage{amssymb}
\usepackage{mathrsfs}
\usepackage{graphicx}
\usepackage[all,cmtip]{xy}
\usepackage[colorlinks,allcolors=red]{hyperref}
\usepackage{array}

\newcommand{\Ain}{\mbox{A$_{\infty}$}}

\newcommand{\Id}{\mbox{Id}}
\newcommand{\aCat}{\mbox{A$_{\infty}$Cat}}
\newcommand{\DgCat}{\mbox{DgCat}}

\newcommand{\Br}{\mbox{B}_{\infty}}
\newcommand{\G}{\mathscr{G}}
\newcommand{\F}{\mathscr{F}}
\newcommand{\A}{\mathscr{A}}
\newcommand{\B}{\mathscr{B}}
\newcommand{\C}{\mathscr{C}}
\newcommand{\Q}{\mathscr{Q}}

\newcommand{\K}{\mathbb{K}}

\newcommand{\E}{\mathscr{E}}
\newcommand{\R}{\mathscr{R}}

\newcommand{\Aalg}{\mbox{A$_{\infty}$Alg}}
\newcommand{\Aalgnu}{\mbox{A$_{\infty}$Alg}^{\mathrm{nu}}}
\newcommand{\Aalgaug}{\mbox{A$_{\infty}$Alg}^{\tiny\text{aug}}}
\newcommand{\Acat}{\mbox{A$_{\infty}$Cat}}
\newcommand{\Acatnu}{\mbox{A$_{\infty}$Cat}^{\mathrm{nu}}}
\newcommand{\Acataug}{\mbox{A$_{\infty}$Cat}^{\tiny\text{aug}}}
\newcommand{\dgcoAlg}{\mbox{DGcoAlg}}
\newcommand{\dgcoCat}{\mbox{DGcoCat}}

\newcommand{\si}{\checkmark}
\newcommand{\no}{$\times$}
\newcommand{\dvuoto}{$\emptyset$}
\newcommand{\ddue}{\raisebox{-2pt}{\xymatrix@C=14pt@R=0pt{\bullet&\bullet}}}
\newcommand{\dpar}{\raisebox{-2pt}{\xymatrix@C=22pt@R=0pt{\bullet\ar@<.6ex>[r]\ar@<-.6ex>[r]&\bullet}}}

\newcommand{\titolo}[1]{\emph{#1}}
\newcommand{\ent}[2]{\par\addvspace{2.5pt}\hangindent=1.1em\hangafter=1 $#1$:\ #2}

\theoremstyle{definition}

\newtheorem{namedthmA}{Theorem}

\newtheorem{lemma}{Lemma}[section]
\newtheorem{proposition}{Proposition}[section]
\newtheorem{corollary}{Corollary}[section]
\newtheorem{theorem}{Theorem}[section]
\newtheorem{definition}{Definition}[section]
\newtheorem{example}{Example}[section]

\begin{document}

\title{An example of the non-existence of coequalizers in the category of the $\Ain$-categories}
\author{Mattia Ornaghi}
\thanks{The author was supported by ERC Advanced Grant-101095900-TriCatApp.}
\date{\today}

\address{\parbox{0.9\textwidth}{Universit\`a degli Studi di Milano\\
Dipartimento di Matematica\\
Via Cesare Saldini 50, 20133 Milano, Italy}}

\email{mattia12.ornaghi@gmail.com}

\begin{abstract}
We prove that the categories of strictly unital (resp.\ non unital) $\Ain$-algebras and $\Ain$-categories, 
with strictly unital (resp.\ non unital) $\Ain$-morphisms and $\Ain$-functors have no coequalizers.
\end{abstract}

\maketitle

\section{Introduction}
It is well known that the category $\DgCat$ of DG-categories (resp.\ of DG-algebras), linear over a commutative ring $\K$, has a model structure (see \cite{Tab}) whose weak equivalences are the quasi-equivalences, so that its homotopy theory is well understood. 
In particular, the first requirement in order to have a model structure is that the category be complete and cocomplete.\ 
For this reason the case of the category $\Aalg$ of $\Ain$-algebras with $\Ain$-morphisms is very delicate. 

Let us start by saying what is known and what is not.\ 
If $\K$ is a field, then Lefèvre-Hasegawa \cite[Th\'eor\`eme 1.3.3.1]{LH} endowed $\Aalg$ with a model structure without limits (see also \cite[\S1.1-2.1]{Val}).\ 
In such a structure every object is fibrant and cofibrant, the weak equivalences are the quasi isomorphisms and the fibrations are the degreewise surjections;\ 
the first technical point is the description of the pullbacks along the fibrations. 
There is a similar result for $\Ain$-categories (see \cite[Theorem 2]{CO}), where the fibrations are the $\Ain$-functors $\F$ such that $\F^1$ is a degreewise surjection and $H^0(\F)$ is an isofibration.\ This is not sufficient to provide a complete model structure (without limits) on $\aCat$ and requires
further study.

The question of the existence of equalizers and coequalizers was asked explicitly by Bruno Vallette in \cite[\S4.1]{Val} and, although the construction of the pullbacks is tricky, finding counterexamples is challenging.\ 
In 2019 we found an example of the non existence of the equalizer see \cite[\S1.5]{COS}.\ In this paper we prove:

\begin{namedthmA}\label{teoA}
Given two strictly unital $\Ain$-algebras $\A:=\big(\K\oplus\K a,m^n_{\A}(a,...,a)=0\big)$, with $\deg(a)=0$, and $\B:=\big(\K\oplus\K b_1\oplus\K b_2,m^{n}_{\B}(b_{i_1},...,b_{i_n})=0\big)$, with $\deg(b_1)=0$ and $\deg(b_2)=-1$ and a strictly unital $\Ain$-morphism $\F\colon\A\to\B$, defined as:
\begin{align*}
\F^1(a):=b_1,\quad \F^2(a,a):=b_2,\quad \F^{n\ge3}:=0 \quad \text{and} \quad \G^1(a):=0, \quad \G^{n\ge 2}:=0,
\end{align*}
the coequalizer $(\F,\G)$ does not exist in $\Aalg$, nor in $\aCat$. 
\end{namedthmA}

The proof goes through the bar construction, which identifies the non unital $\Ain$-categories with the DG-cocategories whose underlying graded cocategory is of the form $T^c(\Q)$ for some graded quiver $\Q$ (Proposition \ref{barful}); coequalizers of DG-cocategories are elementary (Lemma \ref{lemcoeq}), and the whole point is whether the resulting quotient $\E$ is of the form $T^c(\Q)$.\ Lemma \ref{ostruzione} turns this into a numerical condition on the graded dimensions of $\E$, which fails in degree $-7$.\ The strictly unital case is then deduced from the functor $U$ of \cite[Proposition 2.1]{COS} and \cite[Proposition 3.5]{COS2}, which is a left adjoint to the inclusion $i\colon\DgCat\hookrightarrow\aCat$.

Note that the category $\Aalg$ does not have coproducts (see \cite[\S4.1]{Val}); this has to do with the tensor product of $\Ain$-algebras \cite{Orn3}.\ 
On the other hand $\aCat$ admits coproducts, which are given by the disjoint union, cf. \cite[Remark 4.1]{Orn2}.\ So a consequence of Theorem \ref{teoA} is that $\aCat$ is not cocomplete. 

The following table collects what is known about the finite limits and colimits in $\Aalg^{\mathrm{nu}}$, $\Aalg$, $\Acat^{\mathrm{nu}}$ and $\Acat$:

\begin{tabular}{>{\centering\arraybackslash}m{1.5cm}|m{6.6cm}|m{6.6cm}}
\textsc{shape} & \textsc{colimit} & \textsc{limit}\\\hline
&&\\[-7pt]
\dvuoto
&\titolo{initial object}
 \ent{\Aalgnu}{\si\ the zero algebra $0$}
 \ent{\Aalg}{\si\ the algebra $\K$}
 \ent{\Acatnu, \Acat}{\si\ the empty $\Ain$-category}
&\titolo{terminal object}
 \ent{\Aalgnu,\ \Aalg}{\si\ the zero algebra $0$}
 \ent{\Acatnu,\ \Acat}{\si\ one object with zero endomorphisms}
\\[8pt]\hline
&&\\[-7pt]
\ddue
&\titolo{coproduct}
 \ent{\Aalgnu}{\no\ already for the two one dimensional trivial algebras, \cite[\S4.1]{Val}}
 \ent{\Aalg}{\no\ the same pair, by the method of Theorem A}
 \ent{\Acatnu,\ \Acat}{\si\ the disjoint union, \cite[Remark 4.1]{Orn2}}
&\titolo{product}
 \ent{\Aalgnu,\ \Aalg}{\si\ $A\oplus B$ with the componentwise $m^n$, the unit is $(1_A,1_B)$ in the unital case}
 \ent{\Acatnu,\ \Acat}{\si\ the product of the sets of objects, the hom-spaces are the direct sums}
\\[8pt]\hline
&&\\[-7pt]
\dpar
&\titolo{coequalizer}
 \ent{\Aalgnu}{\no\ Theorem \ref{teoprincipale}}
 \ent{\Acatnu}{\no\ Proof of Theorem \ref{teoA}}
 \ent{\Aalg,\ \Acat}{\no\ Theorem \ref{teoA}}
&\titolo{equalizer}
 \ent{\Aalgnu,\ \Aalg}{\no\ in general, \cite[\S1.5]{COS}; \si\ under the hypotheses of \cite{LH}}
 \ent{\Acatnu,\ \Acat}{\no\ in general, \cite[\S1.5]{COS}; \si\ under the hypotheses of \cite{CO}}
\end{tabular}
\\

Two observations: first we can see an $\Ain$-category at the $\infty$-level via the $\Ain$-nerve (see \cite{Orn5}), 
so the next step is to define the derived limits and colimits where it is possible. 
To conclude, recall that the category of $\Ain$-categories with strict $\Ain$-functors is complete and cocomplete and has a model structure \cite{Orn4} whose weak equivalences are the quasi-equivalences.\ So it would be interesting to investigate the cylinder object for $\Ain$-categories, since the coproducts exist in the category of $\Ain$-categories.\ Note that the path object and the right homotopy are studied in detail; see \cite{Orn6}.

\subsection*{Acknowledgements}
I thank Alberto Canonaco and Bruno Vallette for sharing their interest in this article.

\section{Background and the obstruction}\label{secnot}
We fix a commutative ring $\K$, we denote by $\mbox{Ch}_\K$ the category of complexes of graded $\K$-modules with the degree zero chain maps.\ 
We use cohomological conventions: differentials have degree $+1$, and we write $M[1]$ for the shift, i.e.\ $\deg(m[1])=\deg(m)-1$ for $m\in M$, so that $(M[k])^j=M^{j+k}$.

\begin{definition}\label{defalg}
A \emph{non unital $\Ain$-category} is a graded quiver $\A$ together with multilinear maps 
\begin{align*}
m^n_{\A}\colon\A(x_{n-1},x_n)\otimes...\otimes\A(x_0,x_1)\to\A(x_0,x_n)[2-n],
\end{align*}
for any $n\ge1$ and $x_0,...,x_n\in\text{Ob}(\A)$, such that
\[
\sum_{m=1}^{n}\ \sum_{d=0}^{n-m}(-1)^{\dagger_d}\, m^{n-m+1}_{\A}\bigl(f_n,\dots,f_{d+m+1},m^m_{\A}(f_{d+m},\dots,f_{d+1}),f_d,\dots,f_1\bigr)=0,
\]
where $\dagger_d=\mbox{deg}(f_d)+\dots+\mbox{deg}(f_1)-d$.
\end{definition}

\begin{definition}\label{defmor}
A \emph{non unital $\Ain$-functor} $\F\colon\A\to\B$ is a family $\F=\{\F^n\}_{n\ge0}$ such that 
$$\F^0\colon\text{Ob}(\A)\to\text{Ob}(\B)$$ 
is a morphism of sets, and $\F^{n\ge1}$ is a map of graded modules 
\begin{align*}
\F^n\colon\A(x_{n-1},x_n)\otimes...\otimes\A(x_0,x_1)\to\B(x_0,x_n)[1-n]
\end{align*}
for any $n\ge1$ and $x_0,...,x_n\in\text{Ob}(\A)$, such that
\begin{align*}
\sum_{r\ge 1}\ \sum_{s_1+...+s_r=n}m_{\B}^{r}\bigl(&\F^{s_r}(f_n,...,f_{n-s_r+1}),...,\F^{s_1}(f_{s_1},...,f_1)\bigr)=\\
=&\sum_{m=1}^{n}\sum_{k=0}^{n-m}(-1)^{\dagger_k}\F^{n-m+1}\bigl( f_n,...,f_{k+m+1},m^m_{\A}(f_{k+m},...,f_{k+1}),f_k,...,f_1\bigr).
\end{align*}
We say that an $\Ain$-functor $\F$ is \emph{strict} if $\F^{n}=0$ for every $n\ge 2$.\
The composition of two $\Ain$-functors is given by the formula:
\begin{equation*}
(\G\cdot\F)^n(f_n,\dots,f_1)=\sum_{i_r+\dots+i_1=n}\G^r\bigl(\F^{i_r}(f_n,\dots,f_{n-i_r}),\dots,\F^{i_1}(f_{i_1},\dots,f_1)\bigr).
\end{equation*}
\end{definition}

We say that $\A$ is \emph{strictly unital} if there exists a closed morphism $1_x\in\A(x,x)$, for any $x\in\text{Obj}(\A)$, such that 
\begin{align*}
m^2_{\A}(1,f)=m^2_{\A}(f,1)=f,\qquad m^{n\not=2}(...,1,...)=0.
\end{align*}
In the same vein, an $\Ain$-functor $\F$ is \emph{strictly unital} if
\begin{align*}
\F^1(1_x)=1_{\F^0(x)},\qquad \F^{n\not=1}(...,1,...)=0.
\end{align*}
A strictly unital $\Ain$-category $\A$ is \emph{augmented} if it is endowed 
with a strict and strictly unital $\Ain$-functor $\epsilon_{\A}\colon\A\to E_{\tiny\text{Ob}(\A)}$ which is the
identity on objects, where $E_{\tiny\text{Ob}(\A)}$ is the graded quiver whose hom-spaces are given
\begin{equation*}
E_{\tiny\text{Ob}(\A)}(x,y):=\begin{cases}
0,\qquad\text{if $x\not=y$}\\
\K,\qquad\text{if $x=y$}.
\end{cases}
\end{equation*}
An $\Ain$-category with one object is called \emph{$\Ain$-algebra}, an $\Ain$-functor between two $\Ain$-algebras is said \emph{$\Ain$-morphism}.\ 
We denote by $\Aalg$, $\Aalgnu$, $\Aalgaug$ (resp. $\Acat$, $\Acatnu$, $\Acataug$) the categories of strictly unital, non unital, augmented $\Ain$-algebras (resp. $\Ain$-categories) with the corresponding morphisms (resp. functors).\ It is not difficult to see that the categories $\Aalgnu$ and $\Aalgaug$ (resp. $\Acatnu$ and $\Acataug$) are equivalent, by $A\mapsto A_{+}:=\K\oplus A$; note that $\Aalgaug$ is a non full subcategory of $\Aalg$.\ A thorough reference for $\Ain$-categories can be found in \cite{Sei}.

\subsection{Bar construction}
We recall some notions of \cite[\S 2.2]{Orn1}, \cite[\S 1.2]{COS} and \cite{COS2}.\
Let $\Q$ be a graded quiver, the (reduced) tensor cocategory $\Br(\Q)$ has the same objects as $\Q$ and the morphisms are given by
\begin{align*}
\Br(\Q)(x,y)&:=T^c\bigl(\Q[1]\bigr)(x,y)\\
&=\Q(x,y)[1]\bigoplus_{n\ge 1} \bigoplus_{z_1,...,z_n\in\tiny\text{Ob}(\Q)}\Q(z_n,y)[1]\otimes \Q(z_{n-1},z_n)[1]\otimes...\otimes\Q(z_1,z_2)[1]\otimes\Q(x,z_1)[1].
\end{align*}
The coproduct $\Delta$ is given by \cite[Example 1.15]{COS}.\ It is not hard to prove that the forgetful functor 
\begin{align*}
|\text{-}|\colon \text{(graded) coCat} \to \text{(graded) Quivers}
\end{align*}
is left adjoint to the functor $T^c$, see \cite[Proposition 1.17]{COS}.\ 
In particular, we have a (natural) bijection:
\begin{equation}\label{f1}
\xymatrix{
\{\mbox{Functors of (graded) CoCategories}\colon\C\to \Br(\Q)\}\ar@_{<->}[r]^-{1:1}&\{\mbox{Functors of (graded) quivers}\colon |\C|\to\Q\}.
}
\end{equation}
On the other hand, we have the bijection of sets:
\begin{equation}\label{f2}
\xymatrix{
\{\Ain\mbox{-structures on }\Q\}\ar@_{<->}[r]^-{1:1}&\{\mbox{DG-structures on }\Br(\Q)\},
}
\end{equation}
see \cite[\S 2.1, Theorem 2.1]{Orn2}.\ 
Given two $\Ain$-categories $\A$ and $\B$, we have:
\begin{equation}\label{f3}
\xymatrix{
\{\Ain\mbox{-functors: }\A\to\B\}\ar@_{<->}[r]^-{1:1}&\{\mbox{Functor of DG-cocategories}\colon \Br(\A)\to\Br(\B)\}.
}
\end{equation}
Note that the DG-functor $\Br(\F)\colon\Br(\A)\to\Br(\B)$ is defined as follows:
\begin{equation}\label{f4}
\Br(\F)\bigl( a_n[1]\otimes...\otimes a_1[1]\bigr)=\sum_{k\ge1}\ \sum_{i_1+\dots+i_k=n}\F^{i_1}(a_n,...,a_{n-i_1+1})[1]\otimes\dots\otimes\F^{i_k}(a_{i_k},...,a_1)[1].
\end{equation}

\begin{proposition}[{\cite[Remark 1.20]{COS}}]\label{barful}
The bar construction 
\begin{align*}
\Br\colon\Acatnu\to\dgcoCat
\end{align*}
is a fully faithful functor.\ The essential image of $\Br$ consists of those DG-cocategories $\C$ isomorphic to $T^c(\A)$ for some $\Ain$-category $\A$.\ 
In particular, it means that the category $\Acatnu$ is equivalent to a full subcategory of DG-cocategories.\ 
Note that we used this characterization in \cite[Theorem 2]{CO} to prove that $\Aalgnu$ admits equalizers under certain hypotheses.
\end{proposition}

\subsection{The obstruction}\label{seccoeq}
In this subsection we give a criterion to prove that a diagram in $\Aalgnu$ does not have a coequalizer.\ 

\begin{lemma}\label{lemcoeq}
Let $\mathsf{F},\mathsf{G}\colon\mathsf{C}\to\mathsf{D}$ be two morphisms in $\dgcoAlg$ and let $W:=\mbox{Im}(\mathsf{F}-\mathsf{G})\subseteq\mathsf{D}$.\
Then the projection $\mathsf{D}\to\mathsf{D}/W$ is the coequalizer of $\mathsf{F}$ and $\mathsf{G}$ in $\dgcoAlg$.
\end{lemma}

\begin{corollary}\label{corrappr}
Let $f,g\colon A\to B$ be two $\Ain$-morphisms and set
\begin{equation}
W:=\mbox{Im}\bigl(\Br(f)-\Br(g)\bigr)\subseteq\Br(B),\qquad \E:=\Br(B)/W .
\end{equation}
Then, for every $\Ain$-algebra $C$, there is a bijection, natural in $C$,
\begin{equation}
\{h\colon B\to C\ \mbox{ in }\Aalgnu\ :\ h\cdot f=h\cdot g\}\ \longleftrightarrow\ \mbox{Hom}_{\tiny\dgcoAlg}\bigl(\E,\Br(C)\bigr).
\end{equation}
\end{corollary}

\begin{proof}
Combine Proposition \ref{barful} with Lemma \ref{lemcoeq}.
\end{proof}

Let $(M,d_M)$ be a complex, we can consider $M$ as a non unital $\Ain$-algebra such that $m^1_M=d_M$ and $m^{n\ge2}_{M}=0$.\ In particular, by (\ref{f1}) we have the natural isomorphism
\begin{equation}\label{lemtest}
\mbox{Hom}_{\tiny\dgcoAlg}\bigl(\mathsf{C},\Br(M)\bigr)\ \longleftrightarrow\ \mbox{Hom}_{\mbox{\tiny Ch}_K}\bigl(|\mathsf{C}|,M[1]\bigr),
\end{equation}
for every $\mathsf{C}\in\dgcoAlg$.

\begin{lemma}\label{ostruzione}
Let $f,g,\E$ be as in Corollary \ref{corrappr} and assume that the coequalizer $\R$ of $f$ and $g$ exists in $\Aalgnu$.\
We have an isomorphism $|\Br(\R)|\cong|\E|$ in $\mbox{Ch}_{\K}$.\ In particular: $\E\cong\bigoplus_{n\ge1}(\R[1])^{\otimes n}$ as graded modules.
\end{lemma}

\begin{proof}
By Corollary \ref{corrappr} there is a bijection 
\begin{equation}\label{equiz}
\mbox{Hom}_{\tiny\dgcoAlg}(\Br(\R),\Br(C))\simeq\Aalgnu(\R,C )\simeq\{ h\colon B\to C \text{ s.t. $h\circ f=h\circ g$} \}\simeq\mbox{Hom}_{\tiny\dgcoAlg}(\E,\Br(C))
\end{equation}
natural in $C\in\Aalg$.\ 
Let $M$ be a chain complex and we take $C=M[-1]$ (considered as non unital $\Ain$-algebra), then by (\ref{equiz}) and (\ref{lemtest}) we have an isomorphism natural in $M$
\begin{equation*}
\mbox{Hom}_{\mbox{\tiny Ch}_{\K}}\bigl(|\Br(\R)|,M\bigr)\cong\mbox{Hom}_{\mbox{\tiny Ch}_{\K}}\bigl(|\E|,M\bigr).
\end{equation*}
So, by the Yoneda lemma, we have $|\Br(\R)|\cong|\E|$.
\end{proof}

\section{The counterexample}\label{secesempio}
Now we prove Theorem \ref{teoA} making use of Lemma \ref{ostruzione}.\ In this subsection $\K$ is a field.

\begin{definition}\label{defesempio}
We denote by $A$ and $B$ the non unital $\Ain$-algebras defined as follows:
\begin{itemize}
\item[1.] $A:= \K a$ such that $\deg(a)=0$, and $m^{n\ge1}_{\A}=0$;
\item[2.] $B:=\K a\oplus\K b$ such that $\deg(b_1)=0$, $\deg(b_2)=-1$, and $m^{n\ge1}_{\B}=0$.
\end{itemize}
We denote by $f,g\colon A\to B$ the $\Ain$-morphisms defined as follows:
\begin{itemize}
\item[1.] $f^1(a)=b_1$, $f^2(a,a)=b_2$ and $f^{n\ge 3}=0$;
\item[2.] $g^{n\ge 1}=0$.
\end{itemize}
\end{definition}

Note that $A_{+}=\A$, $B_{+}=\B$, $f_{+}=\F$ and $g_{+}=\G$ are the data of Theorem \ref{teoA} and $f$ is \emph{not} strict.\ 
We define $w:=a[1]$ and $u_i:=b_i[1]$, so that
\[
A[1]=\K w,\ \deg(w)=-1,\qquad B[1]=\K u_1\oplus \K u_2,\ \deg(u_1)=-1,\ \deg(u_2)=-2,
\]
We write $w|w:=w\otimes w$, so $\Br(A)$ has basis 
\begin{align*}
\{w|...|w:=w^{\otimes n}\}_{n\ge1}\text{, with }\deg(w^{\otimes n})=-n.
\end{align*}
On the other hand, the \emph{words} in the letters $u_1,u_2$ are a basis for $\Br(B)$. 
For a tuple $\underline{i}=(i_1,\dots,i_k)$ with $i_j\in\{1,2\}$ we write
\begin{align*}
u_{\underline i}:=u_{i_1}|\dots| u_{i_k}\in B[1]^{\otimes k}\subseteq\Br(B),\qquad
\deg(u_{\underline i})=-(i_1+\dots+i_k).
\end{align*}
Finally we set
\begin{align*}
\Phi_n:=\# \{\underline i\text{ such that } i_1+\dots+i_k=n\text{ and $i_j\in\{ 1,2\}$}\}.
\end{align*}
We have $\Phi_1=1$, $\Phi_2=2$, note that $\dim_\K\big(\Br(B)^{-n}\big)=\Phi_n$ and $\Phi_n=\Phi_{n-1}+\Phi_{n-2}$.\\
For example $\Br(B)^{-1}:=<u_1>$, $\Br(B)^{-2}=<u_1|u_1,u_2>$, $\Br(B)^{-3}=<u_1|u_1|u_1,u_2|u_1,u_1|u_2>$.\

The following lemma is straightforward.
\begin{lemma}\label{lemW}
Let $f,g$ be as in Definition \ref{defesempio} and let $W=\mbox{Im}(\Br(f)-\Br(g))$ be as in Corollary \ref{corrappr}.\
Then, for every $n\ge1$,
\begin{align*}
\Br(f)\bigl(w^{\otimes n}\bigr)=\sum_{k\ge1}\sum_{\substack{\underline i=(i_1,\dots,i_k)\\ i_1+\dots+i_k=n}}u_{\underline i}\ \in\ \Br(\B)^{-n}.
\end{align*}
Moreover $W=\bigl\langle\, \Br(\F)(w^{\otimes n})\,\bigr\rangle$ and $\dim_\K(W^{-n})=1$ for every $n\ge1$.\ In particular, $\dim_\K(\E^{-n})=\Phi_n-1$.
\end{lemma}

For instance
\begin{align*}
\Br(f)(w)=u_1,\quad \Br(f)(w^{\otimes2})=u_1| u_1+u_2,\quad
\Br(f)(w^{\otimes3})=u_1| u_1| u_1+u_2| u_1+u_1| u_2
\end{align*}
and
\begin{align*}
\Br(f)(w^{\otimes4})=u_1|u_1|u_1|u_1+u_2| u_1| u_1+u_1| u_2| u_1+u_1| u_1| u_2+u_2| u_2.
\end{align*}

\begin{example}\label{esE3}
Let us compute $\dim_\K(\E^{-n})$ for $n=3$: the homogeneous component $\Br(B)^{-3}$ has basis $u_1\otimes u_1\otimes u_1,\ u_2\otimes u_1,\ u_1\otimes u_2$, so $\dim_\K(\Br(B)^{-3})=\Phi_3=3$. On the other hand, $W^{-3}$ is spanned by $u_1|u_1|u_1+u_2|u_1+u_1|u_2$.\
So $\E^{-3}=\Br(B)^{-3}/W^{-3}$ has dimension 2, a basis is given e.g. by $\{ u_2\otimes u_1,u_1\otimes u_2\}$.
\end{example}

\begin{theorem}\label{teoprincipale}
The pair $\F,\G\colon\A\rightrightarrows\B$ of Definition \ref{defesempio} has no coequalizer in $\Aalgnu$.
\end{theorem}

\begin{proof}
Assume that the coequalizer $\R$ exists.\ By Lemma \ref{ostruzione} we have $\E\cong T^c(V)$ as graded modules, where $V:=\R[1]$, 
so it is enough to prove that there is no graded vector space $V$ such that,\ $\dim_\K(T^c(V)^{-n})=\dim_\K(\E^{-n})$ for every $n\ge 1$.\ 
Note that, by Lemma \ref{lemW} the quotient $\E=\Br(\B)/W$ satisfies $\dim_{\K}(\E^{-n})=\Phi_n-1$, for the first $n$ we have:
\[
\begin{array}{c|ccccccc}
n & 1&2&3&4&5&6&7\\\hline
\dim_\K\bigl(\Br(\B)^{-n}\bigr)=\Phi_n & 1&2&3&5&8&13&21\\
\dim_\K\bigl(\E^{-n}\big)=\Phi_n-1 & 0&1&2&4&7&12&20
\end{array}
\]
Since $\dim_\K(\E^{-1})=0$, the module $V$ is concentrated in degrees $\le-2$.\
We determine a homogeneous basis of $V$ degree by degree.\ 
In degree $-2$ and in degree $-3$ there are no words of length $\ge 2$, so $V^{-2}$ and $V^{-3}$ have dimensions $1$ and $2$; we write
\begin{equation*}
V^{-2}=\langle x\rangle,\qquad V^{-3}=\langle y_1,y_2\rangle.
\end{equation*}
In degree $-4$ the only word of length $\ge2$ is $x\otimes x$ so $\dim_{\K}(V^{-4})=4-1=3$, we write $V^{-4}=\langle e_1,e_2,e_3\rangle$.\ 
In degree $-5$ the words of length $\ge2$ are $x\otimes y_j$ and $y_j\otimes x$, so
\begin{equation*}
\dim_\K (V^{-5})=7-4=3,\qquad V^{-5}=\langle w_1,w_2,w_3\rangle .
\end{equation*}
In degree $-6$ the words of length $2$ are $x\otimes e_k$, $e_k\otimes x$ and $y_i\otimes y_j$, that is $3+3+4=10$, and the only word of length $3$ is $x\otimes x\otimes x$ so $\dim_\K (V^{-6})=12-11=1$.\ 
Now we count the words of length at least $2$ in degree $-7$.\ 
We have the words of length $2$ are:
\begin{equation*}
x\otimes w_k,\quad w_k\otimes x\ (k=1,2,3),\qquad y_i\otimes e_j,\quad e_j\otimes y_i\ (i=1,2,\ j=1,2,3),
\end{equation*}
that is $3+3+6+6=18$.\ 
The words of length $3$ of degree $-7$ are:
\begin{equation*}
x\otimes x\otimes y_j,\qquad x\otimes y_j\otimes x,\qquad y_j\otimes x\otimes x\ (j=1,2),
\end{equation*}
that is $6$.\ 
So we get a subspace of dimension $18+6=24$ of $T^c(V)^{-7}$ but
\begin{equation*}
\dim_\K \big(T^c(V)^{-7}\big)=\dim_\K(\E^{-7})=20<24 ,
\end{equation*}
and we get a contradiction.\ 
\end{proof}

\begin{proof}[Proof of Theorem \ref{teoA}]
First we note that every epimorphism in $\F\colon\A\to\B$ in $\aCat$ is such that $\F^0$ is surjective on the objects, so by Theorem \ref{teoprincipale} we have that $f,g\colon A\rightrightarrows B$ has no coequalizer in the category of non unital $\Ain$-categories.
In particular considering the augmentations we have $f_+=\F$, $g_+=\G$, $A_+=\A$ and $B_+=\B$, so 
\begin{equation}\label{enzo}
\F,\G\colon \A\rightrightarrows \B
\end{equation}
does not have coequalizer in the category of augmented $\Ain$-categories which is a subcategory of $\aCat$. 
To prove that (\ref{enzo}) does not have coequalizer in $\aCat$ or in $\Aalg$ one can use the functor $U\colon\aCat\to\DgCat$ of \cite[Proposition 2.1]{COS}. Since $U$ is left adjoint to the inclusion $i\colon\DgCat\to\aCat$ one can use the same strategy of Theorem \ref{teoprincipale} to show that it cannot exist an $\aCat$ $Q$ (actually a DG-quiver) such that 
\begin{align*}
U(Q)\simeq U(\B)\Bigl/\bigl\langle\mbox{Im}\bigl(U(\F)-U(\G)\bigr)\bigr\rangle\ =\ \Omega(\E)^{+}=T\bigl(|\E|[-1]\bigr)
\end{align*}
as graded vector spaces.
\end{proof}

\begin{example}
The counterexample of Theorem \ref{teoA} seems to be the easiest possible. 
Consider the following diagram
\begin{equation}\label{erzo}
\F',\G'\colon \A\rightrightarrows \B'
\end{equation}
where $\B'$ is the strictly unital $\Ain$-algebra such that:
\begin{equation*}
|\B'|:=\K[\lambda]/(\lambda^2),\quad\text{$\deg(\lambda)=-1$}, \quad\text{ $m^n(\lambda,...,\lambda)=0$, for any $n\ge1$}.
\end{equation*}
The strictly unital $\Ain$-morphisms $\F'$ and $\G'$ are such that the only nonzero components are $\F'^1(1)=1$, $\F'^2(a,a)=\lambda$ and $\G'^1(1)=1$.
It is not hard to prove that the coequalizer of (\ref{erzo}) is given by $(\epsilon_{\B'}\colon\B'\to\K,\K)$ in $\Aalg$ and in $\Acat$.\

Consider now the diagram:
\begin{equation}\label{erzo2}
\F'',\Id_{\A}\colon \A\rightrightarrows \A
\end{equation}
where $\F''$ is the strict strictly unital $\Ain$-morphism such that $\F''^1(a)=-a$ and assume that $\text{char}(\K)\not=2$.\
The coequalizer of (\ref{erzo2}) in $\DgCat$ (the category of DG-categories) is $\K$ but the coequalizer in $\Acat$ is $(\F',\B')$.
\end{example}

\end{document}